\documentclass[a4paper,11pt]{amsart}
\usepackage{amsmath}
\usepackage{cases}
\usepackage{amsfonts}
\usepackage[colorlinks,linkcolor=blue,citecolor=blue]{hyperref}
\usepackage{latexsym, amssymb, amsmath, amsthm, bbm}
\usepackage[all]{xy}
\usepackage{pgfplots}
\usepackage{mathrsfs}
\usepackage [latin1]{inputenc}

\DeclareSymbolFont{EulerExtension}{U}{euex}{m}{n}
\DeclareMathSymbol{\euintop}{\mathop} {EulerExtension}{"52}
\DeclareMathSymbol{\euointop}{\mathop} {EulerExtension}{"48}

\allowdisplaybreaks[4]

\def \id{\operatorname{Id}}
\def \ker{\operatorname{Ker}}

\def \Hom{\operatorname{Hom}}

\def \Id{\operatorname{Id}}

\def \id{\operatorname{Id}}
\def \ker{\operatorname{Ker}}

\numberwithin{equation}{section}

\newtheorem{theorem}{Theorem}[section]
\newtheorem{lemma}[theorem]{Lemma}
\newtheorem{proposition}[theorem]{Proposition}
\newtheorem{corollary}[theorem]{Corollary}
\newtheorem{definition}[theorem]{Definition}
\newtheorem{example}[theorem]{Example}
\newtheorem{remark}[theorem]{Remark}

\newtheorem{convention}[theorem]{Convention}

\begin{document}
\title{Quasitriangular Hopf algebras of dimension $pq^2$}
\thanks{$^\dag$Supported by NSFC 11722016.}


\author{Kun Zhou and Gongxiang Liu}
\address{Department of Mathematics, Nanjing University, Nanjing 210093, China} \email{dg1721021@smail.nju.edu.cn gxliu@nju.edu.cn}
\date{}
\maketitle
\begin{abstract}  Let $p$ and $q$ be distinct odd primes and assume $\Bbbk$ is an algebraically closed field of characteristic zero. We classify all
quasitriangular Hopf algebras of dimension $pq^2$ over $\Bbbk$, which are not simple as Hopf algebras. Moreover, we obtained all quasitriangular structures on these Hopf algebras.
\end{abstract}

\section{Introduction}
Quasitriangular Hopf algebras were introduced by Drinfeld \cite{VG} to give solutions of the quantum Yang-Baxter equations. By definition, a quasitriangular Hopf algebra is a Hopf algebras whose finite-dimensional representations form a braided rigid tensor category, which naturally relates to low dimensional topology (see \cite{AR,LH,SS,NV}). Since then, they have been intensively studied.

The classification of finite-dimensional Hopf algebras has attracted the attention
of many mathematicians in recent years. And the classification of finite-dimensional quasitriangular Hopf algebras can be regarded as an important step towards the classification of all finite-dimensional Hopf algebras: indeed, if $H$ is a finite-dimensional Hopf algebra its \emph{Drinfeld double}, $D(H)$, is a quasitriangular Hopf algebra endowed with a Hopf algebra inclusion $H\hookrightarrow D(H)$.

Let $p$ be a prime number. Hopf algebras of dimension $p$ over an algebraically closed field $\Bbbk$ of characteristic zero were shown by Zhu \cite{ZY} to be isomorphic to the group algebra $\Bbbk[\mathbb{Z}_p]$. For a Hopf algebra $H$ over $\Bbbk$ whose dimension is a product of two primes, there are also many results: If $H$ is Hopf algebra  with dimension $p^2$, then $H$ is isomorphic to group algebra or the Taft algebra of dimension $p^2$ (see \cite{MA1,NS}); Now let $p,q$ be distinct primes. If $p=2$, Ng (see \cite{Ng}) proved that $H$ must be semisimple. It is shown by Etingof and Gelaki (see \cite{PS1}) that semisimple Hopf algebra of dimension $pq$ is always trivial, that is, they are isomorphic to group algebras or the duals of group algebras. In general, it is still an open question whether every Hopf algebra with dimension $pq$ is semisimple or not although it is widely believed that this should be true. For the quasitriangular case, S. Natale has showed that all quasitriangular Hopf algebras of dimension $pq$ with $p,q$ are odd primes, are always semisimple and isomorphic to  group algebras (see \cite{Na1}).

The classification of a Hopf algebra whose dimension is three product of primes in general is also open. The study of Hopf algebras of dimension $p^3$ was carried out by Garc\'{i}a \cite{AG} and he proved that the only ribbon Hopf algebras of dimension $p^3$ are group algebras and Frobenius-Lusztig kernels; The Hopf algebras with dimension $2p^2$ ($p$ is an odd prime) has been researched by several experts: Andruskiewitsch and Natale (see \cite{AN}) classified all such non-semisimple pointed Hopf algebras, and Hilgemann-Ng (see \cite{MN}) proved that if a Hopf algebra $H$ of dimension $2p^2$ is not semisimple, then $H$ or $H^{\ast}$ must be pointed. Cheng-Ng (see \cite{CY}) considered  Hopf algebras with dimension $4p$ and proved that every such non-semisimple Hopf algebra $H$  is pointed if and only if $H$ satisfies that $|G(H)|>2$. The non-simple and semisimple Hopf algebras of dimension $pq^2$, where $p,q$ are distinct primes, were classified in \cite{Na2}.

In this paper, we study non-simple quasitriangular Hopf algebras of dimension $pq^2$, where $p,q$ are distinct odd primes. At first, we give a complete list of them.
\begin{theorem}\label{thm1.1}
Let $H$ be a non-simple Hopf algebra of dimension $pq^2$ over $\Bbbk$. Assume that $H$ admits a quasitriangular structure, then $H$ is semisimple and is isomorphic to one of the following Hopf algebras:
\begin{itemize}
  \item[(i)] a group algebra;
  \item[(ii)] either $\mathscr{A}_0$ or $\mathscr{B}_0,\mathscr{B}_{\lambda_j}$ \emph{(}see Examples \ref{ex2.1.5}-\ref{ex2.1.6} for their definitions\emph{)} for $1\leq j \leq (p-1)/2$.
\end{itemize}
\end{theorem}

Then we determine all possible quasitriangular structures on them: All possible quasitriangular structures on a group algebra of dimension $pq^2$ is given by Propositions \ref{pro3.1.y} and \ref{pro3.1.z}; For Hopf algebras  $\mathscr{A}_0$ and $\mathscr{B}_0,\mathscr{B}_{\lambda_j}\;(1\leq j\leq (p-1)/2)$, the corresponding result is given by the next theorem.
\begin{theorem}\label{thm1.2}
With related notations given in Section $2$, we have the following:
\begin{itemize}
             \item[(i)] All the braiding structures on $\mathscr{A}_0$ are given by $\langle e_hg^i,e_k g^j\rangle=\delta_{h,b^j}\delta_{k,b^{-i}}\lambda^{ij}$, where $\lambda\in \Bbbk$ such that $\lambda^q=1$ and $h,k\in \mathbb{Z}_p\rtimes \mathbb{Z}_q$, $0\leq i,j \leq q-1$;
              \item[(ii)] All the quasitriangular structures on $\mathscr{B}_{\lambda}$, $\lambda\in \{0,\lambda_j|\;1\leq j \leq (p-1)/2\}$, are given by $$R=\sum\limits_{\begin{subarray}{l}  0\leq i,j \leq q \\
                             0\leq k,l \leq q  \\
        \end{subarray}}w(a^ib^j,a^kb^l)e_{a^ib^j}\otimes e_{a^kb^l},$$ where $w$ is a bicharacter on $\mathbb{Z}_q \times \mathbb{Z}_q$ and is determined by the following conditions
         \begin{align*}
w(a,a)=w(a^m,a^m), \;w(a,b)=w(a^m,b^{m^{\lambda}})\eta(a^m,b^{m^{\lambda}},g),\\
w(b,a)=w(b^{m^{\lambda}},a^m)\eta(b^{m^{\lambda}},a^m,g),\;w(b,b)=w(b^{m^{\lambda}},b^{m^{\lambda}}).
\end{align*}
\end{itemize}
\end{theorem}
Note that $\mathscr{A}_0$ is self-dual as a Hopf algebra, a braiding structure on $\mathscr{A}_0$ corresponds to a  quasitriangular structure on $\mathscr{A}_0$ and vise versa.

The paper is organized as follows. Section 2 is devoted to give some notation and preliminary results. In Section 3, we will prove that all non-simple quasitriangular Hopf algebras of dimension $pq^2$ are semisimple. Combining with Natale's (\cite{Na2}) result,
we get an appropriate range of these Hopf algebras now. The last section concentrates on the research of possible quasitriangular structures on them.

\begin{convention}\emph{Throughout the paper we work over an algebraically closed field $\Bbbk$ of characteristic 0. All Hopf algebras in this paper are finite dimensional. For the symbol $\delta$, we mean the classical Kronecker's symbol.  Our references for the theory of Hopf algebras are \cite{MS,R}. For a Hopf algebra $H$, the antipode of $H$ will denoted by $S$. We shall also use the
notation $S_H$ when special emphasis is needed. For a Hopf algebra $H$, the group of group-like elements in $H$ will be denoted by $G(H)$. A Hopf algebra $H$ is called simple if it contains no proper normal Hopf subalgebras in the sense of \cite[3.4.1]{MS}; $H$ is called semisimple if
it is semisimple as an algebra, and cosemisimple if it is cosemisimple as a coalgebra.}
\end{convention}

\section{Preliminaries}
We collect some necessary notions and results in this section.
\subsection{Quasitriangular Hopf algebra.} Recall that a quasitriangular Hopf algebra is a pair $(H, R)$ where $H$ is a Hopf algebra over $\Bbbk$ and $R=\sum R^{(1)} \otimes R^{(2)}$ is an invertible element in $H\otimes H$ such that
\begin{equation*}
 (\Delta \otimes \id)(R)=R_{13}R_{23},\; (\id \otimes \Delta)(R)=R_{13}R_{12},\;\Delta^{op}(h)R=R\Delta(h),
 \end{equation*}
for $h\in H$. Here by definition $R_{12}= \sum R^{(1)} \otimes R^{(2)}\otimes 1 $ and similarly for $R_{13}$ and $R_{23}$. The element $R$ is called a universal $\mathcal{R}$-matrix of $H$ or a \emph{quasitriangular structure} on $H$. Dually, the definition of coquasitriangular Hopf algebra can be given as follows. A coquasitriangular Hopf algebra is a pair $(H,\langle,\rangle)$ where $H$ is a Hopf algebra over $\Bbbk$ and $\langle,\rangle:H\otimes H\rightarrow \Bbbk$ is invertible in the convolution algebra $\Hom_\Bbbk(H\otimes H,\Bbbk)$ satisfying
$$\langle ab,c\rangle=\langle a,c_{(1)}\rangle \langle b,c_{(2)}\rangle,\;
\langle a,bc\rangle=\langle a_{(1)},c\rangle \langle c_{(2)},b\rangle,$$
$$\langle a_{(1)},b_{(1)}\rangle a_{(2)}b_{(2)}=\langle a_{(2)},b_{(2)}\rangle b_{(1)}a_{(1)},$$
for $a,b,c\in H$. The linear map $\langle,\rangle$ is called a \emph{braiding structure} on $H$.

For a quasitriangular Hopf algebra $(H,R)$, there are Hopf algebra maps $l_R: H^{\ast cop}\rightarrow H$ and $r_R: H^{\ast op}\rightarrow H$, given respectively by
$$l_R(f):=(f\otimes \Id)(R),\;\;r_R(f):=(\Id \otimes f)(R),\;f\in H^\ast.$$
Let $H_l$ and $H_r$ denote, respectively, the images of $l_R$ and of $r_R$. Then $H_l$ and $H_r$ are Hopf subalgebras of $H$ of dimension $n>1$, unless $H$ is cocommutative and $R=1\otimes 1$. By \cite[Proposition 2]{R2}, we have $H_l^{\ast cop}\cong H_r$. Since we will discuss the semisimplicity of Hopf algebras, we recall the following theorem which will be used tacitly in the remaining discussion, and the proof of it can be seen in \cite{LR1,LR2}.

\begin{lemma}\label{thm2.1}
The following statements on a finite-dimensional Hopf algebra $H$ over an
algebraically closed field of characteristic zero with antipode $S$ are equivalent:
\begin{itemize}
  \item[(i)] $H$ is semisimple;
  \item[(ii)]  $H$ is cosemisimple;
  \item[(iii)] $\emph{Tr}(S^2)\neq 0$;
  \item[(iv)] $S^2=\Id_H$.
\end{itemize}
\end{lemma}

Let $H_R\subseteq H$ be the minimal quasitriangular Hopf subalgebra of $H$ corresponding to $R$ (see \cite{R2}). It is proved in \cite[Theorem 2]{R2} that $H_R$ is a quotient of $D(H_l)$. In particular, if $H_l$ is semisimple, then $H_r$ is semisimple, and therefore $H_R$ is semisimple. The following result is just Theorem 1.3.5  in \cite{Ge1}.

\begin{lemma}\label{thm2.2}
Let $(H, R)$ be a quasitriangular Hopf algebra over $\Bbbk$. If $H_R$ is semisimple, then $S_H^4 =\Id$.
\end{lemma}
Combining above results, we get the following observation.
\begin{lemma}\label{lem2.3}
Let $(H, R)$ be an odd-dimensional quasitriangular Hopf algebra over $\Bbbk$. Then the following statements are equivalent:
\begin{itemize}
  \item[(i)] $H$ is semisimple;
  \item[(ii)]  $H_R$ is semisimple;
  \item[(iii)] $H_l$ is semisimple;
  \item[(iv)] $H_r$ is semisimple.
\end{itemize}
\end{lemma}

\begin{proof}
Since a Hopf subalgebra of a semisimple Hopf algebra is semisimple, we know (i) $\Rightarrow$ (ii) and (ii) $\Rightarrow$ (iii). By Lemma \ref{thm2.1}, we have (iii) $\Leftrightarrow$ (iv). If $H_r$ is semisimple, then $H_R$ is semisimple. Thus one can use Lemma \ref{thm2.2} to get $S_H^4=\Id$. By assumption, the dimension of $H$ is odd, therefore $\text{Tr}(S_H^2)\neq 0$. Thus, by Lemma \ref{thm2.1}, $H$ is semisimple.
\end{proof}
\subsection{Hopf exact sequence.}

\begin{definition}\label{def2.1.1}
A short exact sequence of Hopf algebras is a sequence of Hopf algebras
and Hopf algebra maps
\begin{equation}\label{ext}
\;\; K\xrightarrow{\iota} H \xrightarrow{\pi} \overline{H}
\end{equation}
such that
\begin{itemize}
  \item[(i)] $\iota$ is injective,
  \item[(ii)]  $\pi$ is surjective,
  \item[(iii)] $\ker(\pi)= HK^+$, $K^+$ is the kernel of the counit of $K$.
\end{itemize}
\end{definition}

Take an exact sequence \eqref{ext}, then $K$ is a normal Hopf
subalgebra of $H$. Conversely, if $K$ is a normal Hopf subalgebra of a Hopf algebra
$H$, then the quotient coalgebra $H=H/HK^+=H/K^+H$ is a quotient Hopf algebra
and $H$ fits into an extension \eqref{ext}, where $\iota$ and $\pi$ are the canonical maps. The following lemma is shown in \cite[Lemma 3.2]{Na1}.
\begin{lemma}\label{lem2.4}
Let the sequence \eqref{ext} be a short exact sequence of finite-dimensional Hopf algebras. Let also $L\subseteq H$ be a Hopf subalgebra. If $L$ is simple, then either $L\subseteq \iota (K)$ or
$L\cap \iota(K)=\Bbbk 1$. In the last case, the restriction $\pi|_L$ is injective.
\end{lemma}

The following two results are \cite[Theorem 6]{PS1} and \cite[Theorem 1.1]{Na1} respectively which will play crucial role in the proof of Proposition \ref{pro3.1}.

\begin{lemma}\label{thm2.4}
Let $H$ be a semisimple Hopf algebra of dimension $pq$ over $\Bbbk$,
where $p$ and $q$ are distinct prime numbers. Then $H$ is trivial.
\end{lemma}

\begin{lemma}\label{thm2.5}
Let $H$ be a Hopf algebra of dimension $pq$ over $\Bbbk$ where $p$ and $q$ are odd prime numbers. Assume that $H$ admits a quasitriangular structure. Then $H$ is semisimple and isomorphic to a group
algebra $\Bbbk F$, where $F$ is a group of order $pq$.
\end{lemma}

An extension \eqref{ext} above such that $K$ is commutative and $\overline{H}$ is cocommutative is called \emph{abelian}. In this situation, we know the extension \eqref{ext} can be written in the following form:
\begin{equation*}
\;\; \Bbbk^G\xrightarrow{\iota} H \xrightarrow{\pi} \Bbbk F,
\end{equation*}
where $G, F$ are finite groups. Abelian extensions were classified by Masuoka
(see \cite[Proposition 1.5]{M3}), and the above $H$ can be expressed as $\Bbbk^G\#_{\sigma,\tau}\Bbbk F$.
To give the description of $\Bbbk^G\#_{\sigma,\tau}\Bbbk F$, we need the following data
\begin{itemize}
\item[(i)] A matched pair of groups, i.e. a quadruple $(F,G,\triangleleft,\triangleright)$, where $G\stackrel{\triangleleft}{\leftarrow}G\times F \stackrel{\triangleright }{\rightarrow}F$ are action of groups on sets, satisfying the following conditions
    \begin{align*}
    g\triangleright(ff')=(g\triangleright f)((g\triangleleft f)\triangleright f'),\quad (gg')\triangleleft f=(g\triangleleft(g'\triangleright f))(g'\triangleleft f),
    \end{align*}
    for $g,g'\in G$ and $f,f'\in F$.
\item[(ii)] $\sigma:G\times F\times F\rightarrow \Bbbk^\times$ is a map such that
\begin{align*}
    \sigma(g\triangleleft f,f',f'')\sigma(g,f,f'f'')=\sigma(g ,f,f')\sigma(g,ff',f'')
    \end{align*}
    and $\sigma(1,f,f')=\sigma(g,1,f')=\sigma(g,f,1)=1$, for $g\in G$ and $f,f',f''\in F$.
\item[(iii)] $\tau:G\times G \times F \rightarrow \Bbbk^\times$ is a map satisfying
    \begin{align*}
    \tau(gg',g'',f)\tau(g,g',g''\triangleright f)=\tau(g',g'',f)\tau(g,g'g'',f)
    \end{align*}
    and $\tau(g,g',1)=\tau(g,1,f)=\tau(1,g',f)$, for $g,g',g''\in G$ and $f\in F$. Moreover, the $\sigma, \tau$ satisfy the following compatible condition
    \begin{align*}
    \sigma(gg',f,f')\tau(g,g',ff')&=\sigma(g,g'\triangleright f, (g'\triangleleft f)\triangleright f')\sigma(g',f,f')\\
    &\tau(g,g',f)\tau(g\triangleleft (g'\triangleleft f),g'\triangleleft f,f'),
    \end{align*}
    for $g,g',g''\in G$ and $f,f',f''\in F$.
\end{itemize}

\begin{definition}\cite[Section 2.2]{AA}\label{def2.1.2}
The Hopf algebra $\Bbbk^G\#_{\sigma,\tau}\Bbbk F$ is equal to $\Bbbk^G\otimes \Bbbk F$ as vector space and we write $a\otimes x$ as $a\#x$. The product, coproduct are given by
\begin{align*}
 &(e_g\#f).(e_{g'}\#f')=\delta_{g\triangleleft f,g'}\;\sigma(g,f,f')\;e_g\#(ff'),\\
 &\Delta(e_g\#f)=\sum_{g'g''=g}\;\tau(g',g'',f)\;e_{g'}\#g''\triangleright f\otimes e_{g''}\#f,
\end{align*}
The unit is $\sum_{g\in G}e_g\#1$ and the counit is $\epsilon(e_g\#f)=\delta_{g,1}$ and the antipode is
\begin{align*}
 S(e_g\#f)=\sigma(g^{-1},g\triangleright f,(g\triangleright f)^{-1})^{-1}\;
 \tau(g^{-1},g,f)^{-1}\;e_{(g\triangleleft f)^{-1}}\#(g\triangleright f)^{-1}.
\end{align*}
\end{definition}
We need the following examples of $\Bbbk^G\#_{\sigma,\tau}\Bbbk F$ for further discussion.
\begin{example}\label{ex2.1.5}
\emph{Let $p, q$ be two prime numbers such that $p\equiv 1(\text{mod}\;q)$ and let $\omega$ be a primitive $q$th root of 1 in $\Bbbk$. Assume $t\in \mathbb{N}$ satisfying $t^q\equiv 1(\text{mod}\;p)$ and $t \not \equiv 1(\text{mod}\;p)$. Let $0\leq l\leq (q-1)$, then the Hopf algebra $\mathscr{A}_l$ \cite[Lemma 1.3.9]{Na2} belongs to $\Bbbk^G\#_{\sigma,\tau}\Bbbk F$. By definition, the data $(G,F,\triangleleft,\triangleright,\sigma,\tau)$ of $\mathscr{A}_l$ is given by the following way
\begin{itemize}
             \item[(i)] $G=\mathbb{Z}_p\rtimes \mathbb{Z}_q=\langle a,b|\;a^p=b^q=1,bab^{-1}=a^t\rangle$,\;$F=\mathbb{Z}_q=\langle g|\;g^q=1\rangle$. The action $\triangleright $ is trivial, and $a\triangleleft g^{i}=a^{t^i},b\triangleleft g^i=b$, for $0\leq i \leq q-1$.
              \item[(ii)] $\sigma(a^i b^j,g^m, g^n)=w^{jlq_{mn}}$, where $q_{mn}$ is the quotient of $m+n$ in the division by $q$ and $1 \leq i\leq p-1$, $0\leq j,m,n \leq q-1$.
              \item[(iii)] $\tau(g ,g',f)=1$ for $g,g'\in G$ and $f\in F$.
\end{itemize}}
\emph{For ease of use, we define a character $\phi:G\rightarrow \Bbbk$ by $\phi(a):=1, \phi(b):=\omega$. Then the above $\sigma$ in Example \ref{ex2.1.5} can be expressed by $\sigma(a^i b^j,g^m, g^n)=\phi^{lq_{mn}}(a^ib^j)$, where $i,j,m,n,q_{mn}$ are the same as (ii).}
\end{example}

\begin{example}\label{ex2.1.6}
\emph{Let $p, q$ be two prime numbers such that $q\equiv 1(\text{mod}\;p)$ and let $\zeta$ be a primitive $q$th root of 1 in $\Bbbk$. Assume $m\in \mathbb{N}$ satisfying $m^p \equiv 1(\text{mod}\;q)$ and $m \not \equiv 1(\text{mod}\;p)$. Let $0\leq \lambda \leq (p-1)$, then the Hopf algebra $\mathscr{B}_\lambda$ \cite[Proposition 1.4.7]{Na2} belongs to $\Bbbk^G\#_{\sigma,\tau}\Bbbk F$. By definition, the data $(G,F,\triangleleft,\triangleright,\sigma,\tau)$ of $\mathscr{B}_\lambda$ is given by the following way
\begin{itemize}
             \item[(i)] $G=\mathbb{Z}_q\times \mathbb{Z}_q=\langle a,b|\;a^q=b^q=1,ab=ba\rangle$,\;$F=\mathbb{Z}_p=\langle g|\;g^p=1\rangle$. The action $\triangleright $ is trivial, and $a\triangleleft g^{-i}=a^{m^i},b\triangleleft g^{-i}=b^{m^{\lambda i}}$, for $0\leq i \leq p-1$.
              \item[(ii)] $\sigma(g ,f,f')=1$ for $g\in G$ and $f,f'\in F$.
              \item[(iii)] $\tau(a^ib^j ,a^kb^l,g^n)=\zeta_{n}^{jk}$, where $\zeta_{n}=\zeta^{c_n(m^{\lambda+1})}$, here $c_n(r):=1+r+...+r^{n-1}$ for $r\in \mathbb{Z}$ and $0\leq i,j,k,l \leq q$, $0\leq n \leq p-1$.
\end{itemize}}
\end{example}

Let $0\leq \lambda \leq (p-1)$. For convenience, we will denote the dual of a Hopf algebra $H$ as $H^\ast$. Since we want to express $\mathscr{B}_\lambda^\ast$ as the form $\Bbbk^G\#_{\sigma,\tau}\Bbbk F$, we give the following lemma.

\begin{lemma}\label{lem2.5}
Assume the data of $\Bbbk^G\#_{\sigma,\tau}\Bbbk F$ is $(G,F,\triangleleft,\triangleright,\sigma,\tau)$. If the action $\triangleright$ is trivial, i.e $g\triangleright f=f$ for $g\in G,f\in F$, then $(\Bbbk^G\#_{\sigma,\tau}\Bbbk F)^\ast \cong \Bbbk^{G'}\#_{\sigma',\tau'}\Bbbk F'$ as a Hopf algebra, where the data $(G',F',\triangleleft',{'\triangleright},\sigma',\tau')$ of $\Bbbk^{G'}\#_{\sigma',\tau'}\Bbbk F'$ is given as follows
\begin{itemize}
             \item[(i)] $G'=F,\;F'=G$. The action $\triangleleft'$ is trivial and $f('\triangleright) g=g\triangleleft f^{-1}$;
              \item[(ii)] $\sigma':G'\times F'\times F'\rightarrow \Bbbk^\times$ is defined by $\sigma'(f,g,g')=\tau(g\triangleleft f^{-1},g'\triangleleft f^{-1},f)$ for $g,g'\in G$ and $f\in F$.
              \item[(iii)] $\tau':G'\times G' \times F' \rightarrow \Bbbk^\times$ is given by $\tau'(f,f',g)=\sigma(g\triangleleft (ff')^{-1},f,f')$ for $g\in G$ and $f,f'\in F$.
\end{itemize}
\end{lemma}

\begin{proof}
For the Hopf algebra $\Bbbk^G\#_{\sigma,\tau}\Bbbk F$, we denote the dual basis of $\{e_g\#f|\;g\in G,f\in F\}$ by $\{E_{g;f}|\;g\in G,f\in F\}$, that is,
$E_{g;f}(e_{g'}\#f')=\delta_{g,g'}\delta_{f,f'}$ for $g,g'\in G$ and $f,f'\in F$. Define $\varphi:(\Bbbk^G\#_{\sigma,\tau}\Bbbk F)^\ast \rightarrow \Bbbk^{G'}\#_{\sigma',\tau'}\Bbbk F'$ by $\varphi(E_{g;f})=e_f\#g\triangleleft f$ for $g\in G$ and $f\in F$. Then we claim $\varphi$ is an isomorphism of Hopf algebras.

By definition, we will get $E_{g;f}E_{g';f'}=\delta_{f,f'}\tau(g,g',f)E_{gg',f}$ for $g,g'\in G$ and $f,f'\in F$. Thus we have
\begin{align}
\label{eq1.1} \varphi(E_{g;f}E_{g';f'})&=\delta_{f,f'}\tau(g,g',f)\varphi(E_{gg',f})\\
\notag &=\delta_{f,f'}\tau(g,g',f)e_f\#(gg')\triangleleft f.
\end{align}
Directly, we have
\begin{align}
\label{eq1.2} \varphi(E_{g;f})\varphi(E_{g';f'})&=(e_f\#g\triangleleft f)(e_{f'}\#g'\triangleleft f')\\
\notag &=\delta_{f\triangleleft'(g\triangleleft f),f'}\sigma'(f,g\triangleleft f,g'\triangleleft f)e_f\#((gg')\triangleleft f)\\
\notag &=\delta_{f,f'}\tau(g,g',f)e_f\#(gg')\triangleleft f.
\end{align}
By the equations \eqref{eq1.1}-\eqref{eq1.2}, we know $\varphi$ is an algebra map. More, it is not hard to see that $\varphi^\ast(E_{f;g})=e_{g\triangleleft f^{-1}}\#f$ for $f\in F$ and $g\in G$, where $\varphi^\ast$ is the dual map of $\varphi$. To show $\varphi$ is coalgebra map, we only need to prove that $\varphi^\ast$ is an algebra map. Similarly, for the Hopf algebra $\Bbbk^{G'}\#_{\sigma',\tau'}\Bbbk F'$, we denote the dual basis of $\{e_f\#g|\;g\in G,f\in F\}$ by $\{E_{f;g}|\;g\in G,f\in F\}$, that is, $E_{f;g}(e_{f'}\#g')=\delta_{f,f'}\delta_{g,g'}$ for $g,g'\in G$ and $f,f'\in F$. Then we know $E_{f;g}E_{f';g'}=\delta_{g\triangleleft f',g'}\sigma(g'\triangleleft(ff')^{-1},f,f')E_{ff';g'}$ for $g,g'\in G$ and $f,f'\in F$. Thus we have
\begin{align}
\label{eq1.3} \varphi^\ast(E_{f;g}E_{f';g'})&=\delta_{g\triangleleft f',g'}\sigma(g'\triangleleft(ff')^{-1},f,f')\varphi^\ast(E_{ff';g'})\\
\notag &=\delta_{g\triangleleft f',g'}\sigma(g'\triangleleft(ff')^{-1},f,f')e_{g'\triangleleft (ff')^{-1}}\#ff'\\
\notag &=\delta_{g,g'\triangleleft (f')^{-1}} \sigma(g'\triangleleft(ff')^{-1},f,f')e_{g'\triangleleft (ff')^{-1}}\#ff'\\
\notag &=\delta_{g,g'\triangleleft (f')^{-1}} \sigma(g\triangleleft f^{-1},f,f')e_{g\triangleleft (f)^{-1}}\#ff'.
\end{align}
Another way, we get
\begin{align}
\label{eq1.4} \varphi^\ast(E_{f;g})\varphi^\ast(E_{f';g'})&=(e_{g\triangleleft f^{-1}}\#f)(e_{g'\triangleleft (f')^{-1}}\#f')\\
\notag &=\delta_{g,g'\triangleleft (f')^{-1}}\sigma(g\triangleleft f^{-1},f,f')e_{g\triangleleft f^{-1}}\#ff'.
\end{align}
Therefore $\varphi$ is coalgebra map by equations \eqref{eq1.3}-\eqref{eq1.4}. By definition, we get $\varphi$ is isomorphism as vector space. Hence we have completed the proof.
\end{proof}

Using this lemma, we get the following characterization of $\mathscr{B}_\lambda^\ast$ (see Example \ref{ex2.1.6}).
\begin{example}\label{ex2.1.7}
\emph{Let $p, q$ be two prime numbers such that $q\equiv 1(\text{mod}\;p)$ and let $\zeta$ be a primitive $q$th root of 1 in $\Bbbk$. Assume $m\in \mathbb{N}$ satisfying $m^p \equiv 1(\text{mod}\;q)$ and $m \not \equiv 1(\text{mod}\;p)$. Let $0\leq \lambda \leq (p-1)$, then $\mathscr{B}_\lambda^\ast$ belongs to $\Bbbk^G\#_{\sigma,\tau}\Bbbk F$. Due to Lemma \ref{lem2.5} and the Example \ref{ex2.1.6}, the data $(G,F,\triangleleft,\triangleright,\sigma,\tau)$ of $\mathscr{B}_\lambda^\ast$ is given by the following way
\begin{itemize}
             \item[(i)] $G=\mathbb{Z}_p=\langle g|\;g^p=1\rangle$,\;$F=\mathbb{Z}_{q}\times \mathbb{Z}_{q}=\langle a,b|\;a^q=b^q=1,ab=ba\rangle$. The action $\triangleleft$ is trivial, and $g^i\triangleright a=a^{m^i},g^i\triangleright b=b^{m^{\lambda i}}$, for $0\leq i \leq p-1$.
              \item[(ii)] $\sigma(g^n,a^ib^j,a^kb^l)=\zeta_{n}^{jkm^{(\lambda+1)n}}$, where $\zeta_{n}=\zeta^{c_n(m^{\lambda+1})}$, here $c_n(r):=1+r+...+r^{n-1}$ for $r\in \mathbb{Z}$ and $0\leq i,j,k,l \leq q$, $0\leq n \leq p-1$.
              \item[(iii)] $\tau(f ,f',g)=1$ for $g\in G$ and $f,f'\in F$.
\end{itemize}}
\end{example}

\begin{remark}\label{rem1.1}\emph{Recall the definition of $\mathscr{B}_\lambda$, we know $\Bbbk G(\mathscr{B}_\lambda)=\Bbbk^G$, where $G=\mathbb{Z}_q\times \mathbb{Z}_q$. By definition, we get $G(\mathscr{B}_\lambda^\ast)$ is the set of characters of $\mathscr{B}_\lambda$. Through a direct calculation, we get $|G(\mathscr{B}_0^\ast)|=pq$ and $|G(\mathscr{B}_\lambda^\ast)|=p$ for $\lambda\neq 0$. Moreover, by Example \ref{ex2.1.7}, we further obtain $\Bbbk G(\mathscr{B}_0^\ast)=\langle e_{g^i}b^j|\;0\leq i \leq p-1,0\leq i \leq q-1\rangle$ as vector space while $\Bbbk G(\mathscr{B}_\lambda^\ast)=\langle e_{g^i}|\;0\leq i \leq p-1\rangle$ as vector space for $\lambda\neq 0$.}\end{remark}

\section{Non-simple quasitriangular Hopf algebras of dimension $pq^2$ is semisimple}\label{sec2.2}

In the following content, we always assume that $p,q$ are distinct odd primes. In this section, we will prove that a non-simple quasitriangular Hopf algebra of dimension $pq^2$ must be semisimple. We start with two basic observations.
\begin{lemma}\label{lem3.0}
Let $(H, R)$ be a quasitriangular Hopf algebra of dimension $pq^2$ over $\Bbbk$. Assume $H$ is not simple as Hopf algebra and $H$ is not semisimple. Then $H$ fits into a short exact sequence \eqref{ext} such that $\emph{dim}(K)\in \{pq,q^2\}$.
\end{lemma}

\begin{proof}
By assumption, $H$ is not simple. Thus we can assume that $H$ fits into a short exact sequence \eqref{ext} and $K\neq \Bbbk 1$ is a proper Hopf subalgebra. Since $(H,R)$ is a quasitriangular Hopf algebra and $\overline{H}$ is a quotient of $H$, we get $\overline{H}$ is a quasitriangular Hopf algebra. By Nichols-Zoeller theorem, we have $\text{dim}(K)\in \{p,q,pq,q^2\}$. To complete the proof, we only need to show $\text{dim}(K)\not\in \{p,q\}$. Assume $\text{dim}(K)=p$, then $K$ is a group algebra \cite{ZY} and $\text{dim}(\overline{H})=q^2$. Therefore $K$ is semisimple Hopf algebra. Due to Lemma \ref{thm2.5}, $\overline{H}$ is also a group algebra. Hence $\overline{H}$ is semisimple Hopf algebra. Since an extension of semisimple Hopf algebras is also semisimple \cite{BM}, we get $H$ is semisimple. But this contradicts our assumption about non-semisimplicity of $H$. Therefore $\text{dim}(K)\neq p$. Similarly, we have $\text{dim}(K)\neq q$.
\end{proof}

\begin{lemma}\label{lem3.0.1}
Let $(H, R)$ be a quasitriangular Hopf algebra of dimension $pq^2$ over $\Bbbk$. Assume $H$ is not simple as Hopf algebra and $H$ is not semisimple. Then $H^{\ast cop}\not\cong H$.
\end{lemma}

\begin{proof}
By Lemma \ref{lem3.0}, we can assume that $H$ fits into a short exact sequence \eqref{ext} such that $\text{dim}(K)\in \{pq,q^2\}$. Assume $H^{\ast cop}\cong H$. Since $(H, R)$ is quasitriangular Hopf algebra, we know $H^{\ast cop}$ is also a quasitriangular Hopf algebra. This implies $H^\ast$ is quasitriangular Hopf algebra. Because $K^\ast$ is a quotient of $H^\ast$, we obtain $K^\ast$ is quasitriangular Hopf algebra. Moreover, $\text{dim}(K^\ast)\in \{pq,q^2\}$. By Lemma \ref{thm2.5}, $K^\ast$ is trivial and hence $K$ is semisimple. However, we know $\overline{H}$ is also semisimple due to the dimensional reason. Because an extension of semisimple Hopf algebras is also semisimple \cite{BM}, we get that $H$ must be a semisimple Hopf algebra. This contradicts our assumption about non-semisimplicity of $H$. Therefore we have $H^{\ast cop}\not\cong H$.
\end{proof}

Now we can get what we want.
\begin{proposition}\label{pro3.1}
Let $(H, R)$ be a quasitriangular Hopf algebra of dimension $pq^2$ over $\Bbbk$. Assume $H$ is not simple as Hopf algebra, then $H$ is semisimple.
\end{proposition}

\begin{proof}
Suppose that $H$ is not semisimple. By Lemma \ref{lem2.3}, $H_R$ is not semisimple. Due to Nichols-Zoeller theorem \cite{NZ} and the classification of Hopf algebras of prime dimension \cite{ZY}, we must have $\text{dim}(H_R)\in \{pq,q^2,pq^2\}$. By Lemma \ref{thm2.5}, we must have $\text{dim}(H_R)=pq^2$. Since $H_R$ is a quotient of $D(H_l)$ \cite[Theorem 2]{R2}, it follows that $\text{dim}(H_R)|\text{dim}(H_l)^2$. Thus $\text{dim}(H_l)\in \{pq,pq^2\}$.

Assume $\text{dim}(H_l)=pq$. By Lemma \ref{lem2.3}, we get $H_l$ is not semisimple. Since the dimensional reason, we obtain $H_l$ must be simple as Hopf algebra. Due to Lemma \ref{lem3.0}, we can assume that $H$ fits into a short exact sequence \eqref{ext} such that $\text{dim}(K)\in \{pq,q^2\}$. Assume $\text{dim}(K)= pq$. Since $H_l$ is simple as Hopf algebra, we can use Lemma \ref{lem2.4} to get $H_l\subseteq K$ or $H_l\cap K=\Bbbk 1$. Since $\text{dim}(H_l)=pq$ and $\text{dim}(\overline{H})=q$, we know $\pi|_{H_l}$ is not injective. By Lemma \ref{lem2.4}, we get $H_l\cap K\neq \Bbbk 1$. Thus $H_l\subseteq K$. Because $\text{dim}(H_l)=\text{dim}(K)=pq$ in this case, we obtain $H_l=K$. Similarly, we have $H_r=K$. Therefore $H_R=K$. Since $\text{dim}(K)=pq$, we have $\text{dim}(H_R)=pq$. By Lemma \ref{thm2.5}, we get $H_R$ is semisimple. But this contradicts with the previous conclusion about $H_R$, so $\text{dim}(K)\neq pq$. Suppose $\text{dim}(K)=q^2$. By Lemma \ref{lem2.4}, $H_l\subseteq K$ or $H_l\cap K=\Bbbk 1$. Since $\text{dim}(H_l)=pq$ and $\text{dim}(K)=q^2$, we get $H_l\neq K$. Hence $H_l\cap K=\Bbbk 1$. This implies that $\pi|_{H_l}$ is injective by Lemma \ref{lem2.4} again. Since $\text{dim}(\overline{H})=p$ and  $\text{dim}(H_l)=pq$, we get $\pi|_{H_l}$ is not injective. Thus $\text{dim}(K)\neq q^2$. Then we get a conclusion that $\text{dim}(H_l)\neq pq$.

Assume $\text{dim}(H_l)=pq^2$. Then we have $H_l=H_r=H$ in this case. Thus $H^{\ast cop}\cong H$. But this can't happen by Lemma \ref{lem3.0.1}. Hence the non-semisimple hypothesis for $H$ is not true.
\end{proof}

\section{Determination of Quasitriangular Structures}
In this section, we will settle the proof of Theorems \ref{thm1.1} and \ref{thm1.2}.
\subsection{An range of Hopf algebras.} The classification of non-trivial semisimple Hopf algebras of dimension $pq^2$ over $\Bbbk$ which are not simple as Hopf algebras has been given in \cite{Na2}. To state this result, we need a set of numbers. Since $p$ is an odd prime, we can find a subset $\{\lambda_i|\;1\leq i \leq (p-1)/2\}$ from the set $\{1,2,...,p-2\}$ such that $\lambda_i\lambda_j\neq 1(\text{mod} \;p)$ if $i\neq j$. Next, we always assume that we have chosen a subset $\{\lambda_i|\;1\leq i \leq (p-1)/2\}$ so that $\lambda_i\lambda_j\neq 1(\text{mod} \;p)$ if $i\neq j$.  For example, if $p=5$ then we can choose $\{1,2\}$ to be a needed subset. The following result is \cite[Theorem 3.12.4]{Na2}.
\begin{lemma}\label{thm3.2}
Suppose $A$ is a non-trivial semisimple Hopf algebras of dimension $pq^2$ over $\Bbbk$ which are not simple as Hopf algebra. Then either
\begin{itemize}
             \item[(i)] $p\equiv 1(\emph{mod} \;q)$, and $A$ is isomorphic to either of the Hopf algebras $\mathscr{A}_l$, $0 \leq l \leq q-1$ in Example \ref{ex2.1.5}, or
              \item[(ii)] $q\equiv 1(\emph{mod} \;p)$, and $A$ is isomorphic to either of the Hopf algebras $\mathscr{B}_{0}$ or $\mathscr{B}_{\lambda_j}$, $1\leq j \leq (p-1)/2$ and their duals. See Example \ref{ex2.1.6}.
\end{itemize}
\end{lemma}

Combining Proposition \ref{pro3.1} and this lemma, we have
\begin{corollary}\label{cor4.2} With notions above and let $H$ be a non-simple quasitriangular Hopf algebra of dimension $pq^2$, then it is isomorphic to one of the following Hopf algebras:
\begin{itemize}
    \item[1)] A group algebra;
    \item[2)] $\mathscr{A}_l$ for $0 \leq l \leq q-1$;
    \item[3)] $\mathscr{B}_{0}$, $\mathscr{B}_{\lambda_j}$, $\mathscr{B}_{0}^{\ast}$, $\mathscr{B}_{\lambda_j}^{\ast}$ for $1\leq j \leq (p-1)/2$.
\end{itemize}
\end{corollary}
It is well-known that a group algebra is cocommutative and thus it always has a trivial quasitriangular structure. Therefore, to complete the classification of non-simple quasitriangular Hopf algebras of dimension $pq^2$, we need consider possible quasitriangular structures on Hopf algebras stated in 2) and 3) of above corollary.
\subsection{Quasitriangular structures on a group algebra of dimension $pq^2$.}
Let $G$ be a finite group. If $G$ is happened to be an abelian group, then  it is well-known that quasitriangular structures are just bicharacters and so we only consider the question for nonabelian groups in this subsection. For the purpose, we first need the following lemma.
\begin{lemma}\label{lem3.y} Let $G$ be a group and  $K\subseteq G$  a normal and commutative subgroup. Assume that $\{e_k\}_{k\in K}$ are the orthogonal idempotents of $\Bbbk [K]$ and $R=\sum_{k,k'\in K}w(k,k')e_k\otimes e_{k'}$, where $w$ is a bicharacter on $K$. For $g\in G$, we denote $ge_{k}g^{-1}=\varphi_g(k)$, $k\in K$. Then $R$ is a quasitriangular structure on $\Bbbk [G]$ if and only if $w(k,k')=w(\varphi_g(k),\varphi_g(k'))$ for $k,k'\in K$ and $g\in G$.
\end{lemma}

\begin{proof}
By a direct computation.
\end{proof}

The following lemma is the result in \cite[Section 4]{CG}. And we try to keep the notation in \cite[Section 4]{CG} as much as possible.
\begin{lemma}\label{lem3.z}
Let $G$ be a group of order $pq^2$\emph{(}$q>p$\emph{)}. Then $G$ is isomorphic to the following cases
\begin{itemize}
  \item[(1)] $\beta_1=\mathbb{Z}_{pq^2}$;
  \item[(2)]  $\beta_2=\mathbb{Z}_{p}\bigoplus\mathbb{Z}_{q}\bigoplus\mathbb{Z}_{q}$;
  \item[(3)] $p\mid (q^2-1)$ and  $\beta_3=\langle s,t|\;s^{q^2}=t^p=1,tst^{-1}=s^m\rangle$, $m\in \mathbb{N}$ satisfying $m^p \equiv 1(\emph{mod}\;q^2)$ and $m \not \equiv 1(\emph{mod}\;q^2)$;
  \item[(4)] $p\mid (q-1)$ and  $\beta_4=\langle s,t,u|\;s^{q}=t^q=u^p=1,tst^{-1}=s,usu^{-1}=s,utu^{-1}=t^m\rangle$, $m\in \mathbb{N}$ satisfying $m^p \equiv 1(\emph{mod}\;q)$ and $m \not \equiv 1(\emph{mod}\;q)$;
  \item[(5)] $p\mid (q-1)$ and $\beta_5=\langle s,t,u|\;s^{q}=t^q=u^p=1,tst^{-1}=s,usu^{-1}=s^m,utu^{-1}=t^m\rangle$, $m\in \mathbb{N}$ satisfying $m^p \equiv 1(\emph{mod}\;q)$ and $m \not \equiv 1(\emph{mod}\;q)$;
  \item[(6)] $p\mid (q-1)$ and $\beta_6=\langle s,t,u|\;s^{q}=t^q=u^p=1,tst^{-1}=s,usu^{-1}=s^m,utu^{-1}=t^n\rangle$, $m\in \mathbb{N}$ satisfying $m^p \equiv 1(\emph{mod}\;q),n^p \equiv 1(\emph{mod}\;q)$ and $m \not \equiv 1(\emph{mod}\;q),n\not \equiv 1(\emph{mod}\;q)$ and $m \not \equiv n(\emph{mod}\;q)$;
  \item[(7)] $p\mid (q+1)$ and $\beta_7=\langle s,t,u|\;s^{q}=t^q=u^p=1,tst^{-1}=s,usu^{-1}=t,utu^{-1}=s^mt^n,u^qsu^{-q}=s\rangle$, $m,n\in \mathbb{N}$.
\end{itemize}
Moreover, if $G$ is not abelian, i.e. $G\in \{\beta_i|\;3\leq i \leq 7\}$, then $G$ has a largest non-trivial commutative normal subgroup, that is, it contains all other non-trivial commutative normal subgroups \emph{(}see next proposition for the construction of this largest subgroup\emph{)}.
\end{lemma}

\begin{proposition}\label{pro3.1.y}
We have
\begin{itemize}
  \item[(1)] All the quasitriangular structures on $\beta_3$ are given by $R=\sum_{g,h\in K}w(g,h)e_g \otimes e_h,$ where $w$ is a bicharacter on $K=\langle s\rangle$ such that $w(s^m,s^m)=w(s,s)$;
  \item[(2)]  All the quasitriangular structures on $\beta_4$ are given by
  $R=\sum_{g,h\in K}w(g,h)e_g \otimes e_h,$ where $w$ is a bicharacter on $K=\langle s,t\rangle$ such that
  $$w(s,t^m)=w(s,t),\;w(t^m,s)=w(t,s),\;w(t^m,t^m)=w(t,t);$$
  \item[(3)] All the quasitriangular structures on $\beta_5$ are given by
  $R=\sum_{g,h\in K}w(g,h)e_g \otimes e_h,$ where $w$ is a bicharacter on $K=\langle s,t\rangle$ such that
  $$w(s^m,s^m)=w(s,s),\;w(s^m,t^m)=w(s,t),$$
  $$\;w(t^m,s^m)=w(t,s),\;w(t^m,t^m)=w(t,t);$$
  \item[(4)] All the quasitriangular structures on $\beta_6$ are given by
  $R=\sum_{g,h\in K}w(g,h)e_g \otimes e_h,$ where $w$ is a bicharacter on $K=\langle s,t\rangle$ such that
  $$w(s^m,s^m)=w(s,s),\;w(s^m,t^n)=w(s,t),$$
  $$\;w(t^n,s^m)=w(t,s),\;w(t^n,t^n)=w(t,t);$$
  \item[(5)] All the quasitriangular structures on $\beta_7$ are given by
  $R=\sum_{g,h\in K}w(g,h)e_g \otimes e_h,$ where $w$ is a bicharacter on $K=\langle s,t\rangle$ such that
  $$w(t,t)=w(s,s),\;w(t,s^mt^n)=w(s,t),$$
  $$\;w(s^mt^n,t)=w(t,s),\;w(s^mt^n,s^mt^n)=w(t,t);$$
\end{itemize}
\end{proposition}

\begin{proof}
By Lemma \ref{lem3.z}, we can assume $K$ is the largest non-trivial commutative normal subgroup of $G$.
Assume $R$ is a quasitriangular structures on $\Bbbk G$, then $H_R\in \Bbbk [K]\otimes \Bbbk [K]$ by \cite[Theorem 3]{AD}.
Due to Lemma \ref{lem3.y}, we get what we want.
\end{proof}

Assume $q<p$, then the following lemma is the result in \cite[Section 5]{CG}. And we also keep the notation in \cite[Section 5]{CG} as much as possible.
\begin{lemma}\label{lem3.w}
Let $G$ be a group of order $pq^2$($q<p$). Then $G$ is isomorphic to the following cases
\begin{itemize}
  \item[(1)] $\gamma_1=\mathbb{Z}_{pq^2}$;
  \item[(2)] $\gamma_2=\mathbb{Z}_{p}\bigoplus\mathbb{Z}_{q}\bigoplus\mathbb{Z}_{q}$;
  \item[(3)] $q\mid (p-1)$ and $\gamma_3=\langle s,t|\;s^{p}=t^{q^2}=1,tst^{-1}=s^m \rangle$, $m\in \mathbb{N}$ satisfying $m^q \equiv 1(\emph{mod}\;p)$ and $m \not \equiv 1(\emph{mod}\;p)$;
  \item[(4)] $q^2\mid (p-1)$ and $\gamma_4=\langle s,t|\;s^{p}=t^{q^2}=1,tst^{-1}=s^m \rangle$, $m\in \mathbb{N}$ satisfying $m^{q^2} \equiv 1(\emph{mod}\;p)$ and $m^q \not \equiv 1(\emph{mod}\;p)$;
  \item[(5)] $q\mid (p-1)$ and $\gamma_5=\langle s,t,u|\;s^{p}=t^q=u^q=1,tut^{-1}=u,tst^{-1}=s,usu^{-1}=s^m\rangle$, $m\in \mathbb{N}$ satisfying $m^q \equiv 1(\emph{mod}\;p)$ and $m \not \equiv 1(\emph{mod}\;p)$;
  \item[(6)] $q\mid (p-1)$ and $\gamma_6=\langle s,t,u|\;s^{p}=t^q=u^q=1,tut^{-1}=u,tst^{-1}=s^m,usu^{-1}=s^n\rangle$, $m,n\in \mathbb{N}$ satisfying $m^q \equiv n^q \equiv 1(\emph{mod}\;p)$ and $m \not \equiv 1(\emph{mod}\;p),n \not \equiv 1(\emph{mod}\;p)$ and   $m \not \equiv n(\emph{mod}\;p)$;
\end{itemize}
Moreover, if $G$ is not abelian, i.e. $G\in \{\gamma_i|\;1\leq i \leq 6,i\neq 1,2\}$, then $G$ has a largest non-trivial commutative and normal subgroup.
\end{lemma}

\begin{proposition}\label{pro3.1.z}
We have
\begin{itemize}
  \item[(1)] All the quasitriangular structures on $\gamma_3$ are given by $R=\sum_{g,h\in K}w(g,h)e_g \otimes e_h,$ where $w$ is a bicharacter on $K=\langle s,t^q\rangle$ such that
  $$w(s^m,s^m)=w(s,s),\;w(s^m,t^q)=w(s,t^q),\;w(t^q,s^m)=w(t^q,s);$$
  \item[(2)]  All the quasitriangular structures on $\gamma_4$ are given by
  $R=\sum_{g,h\in K}w(g,h)e_g \otimes e_h,$ where $w$ is a bicharacter on $K=\langle s\rangle$ such that
  $w(s^m,s^m)=w(s,s)$;
  \item[(3)] All the quasitriangular structures on $\gamma_5$ are given by
  $R=\sum_{g,h\in K}w(g,h)e_g \otimes e_h,$ where $w$ is a bicharacter on $K=\langle s,t\rangle$ such that
  $$w(s^m,s^m)=w(s,s),\;w(s^m,t)=w(s,t),\;w(t,s^m)=w(t,s);$$
  \item[(4)] All the quasitriangular structures on $\gamma_6$ are given by
  $R=\sum_{g,h\in K}w(g,h)e_g \otimes e_h,$ where $w$ is a bicharacter on $K=\langle s\rangle$ such that
  $$w(s^m,s^m)=w(s,s),\;w(s^n,s^n)=w(s,s).$$
\end{itemize}
\end{proposition}

\begin{proof}
Similar to the proof of Proposition \ref{pro3.1.y}.
\end{proof}

\subsection{Quasitriangular structures on $\mathscr{A}_l$.} At first, Natale (see \cite[Lemma 1.3.9]{Na2}) already proved that $\mathscr{A}_l$ are self dual for $0 \leq l \leq q-1$. Therefore, it is enough to consider the braiding structures on $\mathscr{A}_l$ instead of their quasitriangular structures. Secondly, we need point out that Gelaki already constructed the Hopf algebra $\mathscr{A}_1$ in \cite[Theorem 3.11]{Ge1} and showed $\mathscr{A}_1$ admits no quasitriangular structure. As the first step to our aim, we will show a more general result than Gelaki's, that is, we will show that all  $\mathscr{A}_l$ for $1\leq l\leq q-1$ admit no quasitriangular structures. To do this, we start with the following two lemmas.
\begin{lemma}\label{lem3.5}
With notions given in Example \ref{ex2.1.5}, the only braiding structure on $\Bbbk^G$ is trivial, i.e. $\langle e_g,e_h\rangle=\epsilon(e_g)\epsilon(e_h)$, where $g,h\in G$.
\end{lemma}

\begin{proof}
Assume $R$ is a quasitriangular structure on $\Bbbk G$. We will show $R=1\otimes 1$ and thus we complete the proof. By \cite[Theorem 3]{AD}, there are commutative and normal subgroups $H,K\subseteq G$ such that $R\in \Bbbk H\otimes \Bbbk K$. Since $G=\langle a,b\;|a^p=b^q=1,bab^{-1}=a^t\rangle$, the only non-trivial commutative and normal subgroup of $G$ is $\mathbb{Z}_p=\langle a\rangle$. Thus $R\in \Bbbk \mathbb{Z}_p \otimes \Bbbk \mathbb{Z}_p$.

Let $\omega$ be a primitive $p$th root of 1. If we define $e_{a^i}=p^{-1}\sum_{j=1}^p \omega^{ij}a^j$, $0\leq i \leq p-1$, then $\{e_{a^i}\}_{i=0}^{p-1}$ is a basis of orthogonal idempotents of $\Bbbk \mathbb{Z}_p$. Moreover, we have $be_{a^i}b^{-1}=e_{a^{it'}}$ for $0\leq i \leq (p-1)$, where $t':=t^{q-1}$. Since $R\in \Bbbk \mathbb{Z}_p \otimes \Bbbk \mathbb{Z}_p$, we can write $R=\sum_{i,j=0}^{p-1}w(a^i,a^j)e_{a^i}\otimes e_{a^j}$, where $w$ is a bicharacter on $\mathbb{Z}_p$. Because $\Delta^{cop}(b)R=R\Delta(b)$, we get $w(a^i,a^j)=w(a^{it},a^{jt})$, $0\leq i,j \leq p-1$. In particular, $w(a,a)=w(a^t,a^t)$. By assumption, $t^q\equiv 1(\text{mod}\;p)$ and $t\not \equiv 1(\text{mod}\;p)$, $q>2$, thus $t^2\not \equiv 1(\text{mod}\;p)$. And this implies $w(a,a)=1$. Since $w$ is bicharacter on $\mathbb{Z}_p$, we have $R=1\otimes 1$.
\end{proof}


\begin{lemma}\label{lem3.6}
With notions given in Example \ref{ex2.1.5} and assume $\langle , \rangle$ is a braiding structure on $\mathscr{A}_l$  for $0 \leq l \leq q-1$. Then there are $g_0,g_1\in G$ such that
\begin{itemize}
             \item[(i)] $\langle e_h, g^i\rangle=\delta_{h,g_0^i}$ and $\langle g^i,e_h\rangle=\delta_{h,g_1^i}$, $0\leq i \leq q$ and $h\in G$;
              \item[(ii)] $\phi^l(g_0)=\phi^l(g_1)$;
              \item[(iii)] $h\triangleleft g=g_0hg_0^{-1}$ and $h\triangleleft g=g_1^{-1}hg_1$, $h\in G$.
\end{itemize}
\end{lemma}

\begin{proof}
Since $\langle ,\rangle$ is a braiding structure and $g$ is a grouplike element in $\mathscr{A}_l$, we know $\langle ,g\rangle:\mathscr{A}_l\rightarrow \Bbbk$ is an algebra map. Thus the restriction $\langle ,g\rangle|_{\Bbbk^G}$ is an algebra map, i.e a character of $\Bbbk^G$. Let $g_0\in G$ such that $\langle e_h,g\rangle=\delta_{h,g_0}$. By definition, $gg^{i}=\sum_{k\in G}\sigma(k,g,g^{i})e_k g^{i+1}$. Thus $\langle e_h,gg^{i}\rangle=\sum_{k\in G}\sigma(k,g,g^{i})\langle e_h,e_k g^{i+1}\rangle$ for $i\geq 0$. Since $\langle ,\rangle$ is a braiding structure, we have $\langle e_h,gg^{i}\rangle=\sum_{rs=h}\langle e_r,g^{i}\rangle \langle e_s,g\rangle=\langle e_{hg_0^{-1}},g^{i}\rangle$. Similarly, we have $\langle e_h,e_k g^{i+1}\rangle=\sum_{rs=h}\langle e_r,g^{i+1}\rangle \langle e_s,e_k\rangle$. By Lemma \ref{lem3.5}, the restriction $\langle ,\rangle|_{\Bbbk^G}$ is trivial. Thus $\langle e_h,e_k g^{i+1}\rangle=\delta_{k,1}\langle e_h,g^{i+1}\rangle$. This implies $\sum_{k\in G}\sigma(k,g,g^{i})\langle e_h,e_k g^{i+1}\rangle=\langle e_h,g^{i+1}\rangle$. Therefore $\langle e_{hg_0^{-1}},g^{i}\rangle=\langle e_h,g^{i+1}\rangle$. By induction way, we have $\langle e_h, g^i\rangle=\delta_{h,g_0^i}$, $0\leq i \leq q$. Similarly, we can find $g_1\in G$ such that $\langle g^i,e_h\rangle=\delta_{h,g_1^i}$, $0\leq i \leq q$. Thus we have (i).

By definition, we have $gg^{i}=\sum_{k\in G}\sigma(k,g,g^{i})e_k g^{i+1}$ for $i\geq 0$. Thus we get $\langle gg^{i},g\rangle =\sum_{k\in G}\sigma(k,g,g^{i})\langle e_kg^{i+1},g\rangle$. Let $\langle g,g\rangle:=\lambda$. Because $\langle ,g\rangle$ is an algebra map, $\langle gg^{i},g\rangle=\lambda \langle g^{i},g\rangle$ and $\langle e_kg^{i+1},g\rangle=\delta_{k,g_0}\langle g^{i+1},g\rangle$. Thus $\langle g^{i+1},g\rangle=\lambda \sigma(g_0,g,g^{i})^{-1} \langle g^{i},g\rangle $. By induction way, we have $\langle g^{q},g\rangle=\lambda^{q}\prod_{i=1}^{q-1}\sigma(g_0,g,g^i)^{-1}$. By definition, we have $\sigma(g_0,g,g^i)=1$ for $1\leq i \leq q-2$ and $\sigma(g_0,g,g^{q-1})=\phi^l(g_0)$. Thus $\langle g^{q},g\rangle=\lambda^{q}\phi^l(g_0)^{-1}$. By definition, $g^q=1$. Therefore $\lambda^{q}=\phi^l(g_0)$. Similarly, we have $\lambda^{q}=\phi^l(g_1)$. And hence we have (ii).

Because $\langle ,\rangle$ is a braiding structure, we get $\sum_{rs=h}\langle e_r,g\rangle e_sg=\sum_{rs=h}ge_r\langle e_s,g\rangle$. Since $\langle e_r,g\rangle=\delta_{r,g_0}$ and $\langle e_s,g\rangle=\delta_{s,g_0}$, we have $e_{g_0^{-1}h}g=ge_{hg_0^{-1}}$. By definition, $e_{g_0^{-1}h}g=ge_{(g_0^{-1}h)\triangleleft g}$. Thus $ge_{(g_0^{-1}h)\triangleleft g}=ge_{hg_0^{-1}}$. Because $gg^{q-1}=\sum_{k\in G}\sigma(k,g,g^{q-1})e_k$, we know $g$ is invertible. Hence $(g_0^{-1}h)\triangleleft g=hg_0^{-1}$. In particular, $g_0^{-1}\triangleleft g=g_0^{-1}$ by letting $h=1$. And hence $h\triangleleft g=g_0hg_0^{-1}$. Similarly, we have $\sum_{rs=h}\langle g,e_r\rangle ge_s=\sum_{rs=h}e_r g\langle g,e_s\rangle$. Since $\langle g,e_r\rangle=\delta_{r,g_1}$ and $\langle g,e_s\rangle=\delta_{s,g_1}$, we have $ge_{g_1^{-1}h}=e_{hg_1^{-1}}g$. Because $e_{hg_1^{-1}}g=ge_{(hg_1^{-1})\triangleleft g}$ and $g$ is invertible, we have $ (hg_1^{-1})\triangleleft g=g_1^{-1}h$. Thus $g_1^{-1}\triangleleft g=g_1^{-1}$ by letting $h=1$. And hence we get $h\triangleleft g=g_1^{-1}hg_1$.
\end{proof}

Now we can use Lemmas \ref{lem3.5}-\ref{lem3.6} to get the following promised result.

\begin{proposition}\label{pro3.7}
The Hopf algebras $\mathscr{A}_l$, $1 \leq l \leq q-1$ admit no quasitriangular structure.
\end{proposition}

\begin{proof}
Let $0 \leq l \leq q-1$. Assume $\langle , \rangle$ is a braiding structure on $\mathscr{A}_l$. Let $S:=\{h\in G|\;h\triangleleft g=h\}$. Directly, we know $S$ is proper subgroup of $G$. Since $b^i\in S$ for $1\leq i \leq q$, we have $S=\{b^i|\;1\leq i \leq q\}$. By (iii) of Lemma \ref{lem3.6}, we have $g_0\in S$. Let $g_0=b^i$, where $1\leq i\leq q$. Because $h\triangleleft g=g_0hg_0^{-1}$, we get $a^t=b^ia b^{-i}$ by letting $h=a$. By definition, $b^ia b^{-i}=a^{t^i}$. Thus $a^t=a^{t^i}$. This implies $t^{i-1}\equiv 1(\text{mod}\;p)$. By assumption, we have $0\leq i-1 \leq q-1$. Thus we have $i=1$. And hence $g_0=b$. Similarly, we get $g_1=b^{-1}$. By (ii) of Lemma \ref{lem3.6}, we have $\phi^l(b)=\phi^l(b^{-1})$. Therefore we get $\omega^{2l}=1$. If $1 \leq l \leq q-1$, then $\omega^{2l}\neq 1$. But this contradicts the previous conclusion, so $\mathscr{A}_l$ admits no quasitriangular structure.
\end{proof}
The last task of this subsection is to find all possible quasitriangular structures on $\mathscr{A}_0$.
\begin{proposition}\label{pro3.8}
All the braiding structures on $\mathscr{A}_0$ are given by $\langle e_hg^i,e_k g^j\rangle=\delta_{h,b^j}\delta_{k,b^{-i}}\lambda^{ij}$, where $\lambda\in \Bbbk$ such that $\lambda^q=1$ and $h,k\in \mathbb{Z}_p\rtimes \mathbb{Z}_q$, $0\leq i,j \leq q-1$.
\end{proposition}

\begin{proof}
Assume $\langle , \rangle$ is a braiding structure on $\mathscr{A}_0$. By Lemma \ref{lem3.5}, we have $\langle e_g,e_h\rangle=\epsilon(e_g)\epsilon(e_h)$ for $g,h\in G$. Similar to the proof of Proposition \ref{pro3.7}, we get $g_0=b$ and $g_1=b^{-1}$. Let $\langle g, g\rangle=\lambda$. Note that $g$ is a grouplike element, and by induction way, we have $\langle g^i, g^j\rangle=\lambda^{ij}$. This implies $\langle g^q, g\rangle=\lambda^{q}$. Because $g^q=1$, we get $\lambda^q=1$. Since $\langle , \rangle$ is a braiding structure, $\langle e_hg^i,e_k g^j\rangle=\sum_{rs=k}\langle e_h,e_rg^j\rangle \langle g^i,e_sg^j\rangle$. Directly, we have $\langle e_h,e_rg^j\rangle=\sum_{uv=h}\langle e_u,g^j\rangle\langle e_v,e_r\rangle=\delta_{r,1}\delta_{h,b^j}$ and $\langle g^i,e_sg^j\rangle=\langle g^i,e_s\rangle\langle g^i,e_sg^j\rangle=\delta_{s,b^{-i}}\lambda^{ij}$. Thus $\langle e_hg^i,e_k g^j\rangle=\delta_{h,b^j}\delta_{k,b^{-i}}\lambda^{ij}$.

Conversely, let $\lambda\in \Bbbk$ such that $\lambda^q=1$ and we can define $\langle e_hg^i,e_k g^j\rangle=\delta_{h,b^j}\delta_{k,b^{-i}}\lambda^{ij}$, where $h,k\in \mathbb{Z}_p\rtimes \mathbb{Z}_q$, $0\leq i,j \leq q-1$. Due to $\lambda^q=1$, we have $\langle e_hg^i,e_k g^j\rangle=\delta_{h,b^j}\delta_{k,b^{-i}}\lambda^{ij}$ for $i,j\in \mathbb{N}$. Next we will show $\langle , \rangle$ is a braiding structure. Firstly, we show $\langle (e_hg^i)(e_k g^j),e_r g^m\rangle=\langle e_hg^i,(e_r g^m)_{(1)}\rangle \langle e_k g^j,(e_r g^m)_{(2)}\rangle$. Since
\begin{align*}
\langle (e_hg^i)(e_k g^j),e_r g^m\rangle&=\langle \delta_{h\triangleleft g^{i},k}e_h g^{i+j},e_r g^m\rangle\\
&=\delta_{h\triangleleft g^{i},k}\delta_{h,b^m}\delta_{r,b^{-(i+j)}}\lambda^{(i+j)m}\\
&=\delta_{h,k}\delta_{h,b^m}\delta_{r,b^{-(i+j)}}\lambda^{(i+j)m}
\end{align*}
and
\begin{align*}
\langle e_hg^i,(e_r g^m)_{(1)}\rangle \langle e_k g^j,(e_r g^m)_{(2)}\rangle&=\sum_{cd=r}\langle e_hg^i,e_c g^m\rangle \langle e_k g^j,e_d g^m\rangle\\
&=\sum_{cd=r}\delta_{h,b^m}\delta_{c,b^{-i}}\lambda^{im} \langle e_k g^j,e_d g^m\rangle\\
&=\sum_{cd=r}\delta_{h,b^m}\delta_{c,b^{-i}}\lambda^{im}\delta_{k,b^m}\delta_{d,b^{-j}}\lambda^{jm}\\
&=\delta_{h,k}\delta_{h,b^m}\delta_{r,b^{-(i+j)}}\lambda^{(i+j)m},
\end{align*}
we have $\langle (e_hg^i)(e_k g^j),e_r g^m\rangle=\langle e_hg^i,(e_r g^m)_{(1)}\rangle \langle e_k g^j,(e_r g^m)_{(2)}\rangle$. Secondly, we will show $\langle e_hg^i,e_k g^je_r g^m\rangle=\langle (e_h g^i)_{(1)},e_r g^m\rangle \langle (e_h g^i)_{(2)},e_k g^j\rangle$. Since
\begin{align*}
\langle e_hg^i,e_k g^je_r g^m\rangle&=\langle e_hg^i,\delta_{k\triangleleft g^{j},r}e_kg^{j+m}\rangle\\
                &=\delta_{k\triangleleft g^{j},r} \delta_{h,b^{m+j}}\delta_{k,b^{-i}}\lambda^{i(j+m)}\\
                &=\delta_{k,r} \delta_{h,b^{m+j}}\delta_{k,b^{-i}}\lambda^{i(j+m)}
\end{align*}
and
\begin{align*}
\langle (e_h g^i)_{(1)},e_r g^m\rangle \langle (e_h g^i)_{(2)},e_k g^j\rangle &=\sum_{cd=h}\langle e_c g^i,e_r g^m\rangle \langle e_d g^i,e_k g^j\rangle\\
                &=\sum_{cd=h}\delta_{c,b^m} \delta_{r,b^{-i}}\lambda^{im} \langle e_d g^i,e_k g^j\rangle\\
                &=\sum_{cd=h}\delta_{c,b^m} \delta_{r,b^{-i}}\lambda^{im}\delta_{d,b^j} \delta_{k,b^{-i}}\lambda^{ij}\\
                &=\delta_{k,r} \delta_{h,b^{m+j}}\delta_{k,b^{-i}}\lambda^{i(j+m)},
\end{align*}
we have $\langle e_hg^i,e_k g^je_r g^m\rangle=\langle (e_h g^i)_{(1)},e_r g^m\rangle \langle (e_h g^i)_{(2)},e_k g^j\rangle$. Finally, we will show the following equation
\begin{align}
\label{eq2.1} \langle (e_hg^i)_{(1)},(e_kg^j)_{(1)}\rangle (e_hg^i)_{(2)}(e_kg^j)_{(2)}=(e_kg^j)_{(1)}(e_hg^i)_{(1)}\langle (e_hg^i)_{(2)},(e_kg^j)_{(2)}\rangle .
\end{align}
Since
\begin{align*}
\langle (e_hg^i)_{(1)},(e_kg^j)_{(1)}\rangle (e_hg^i)_{(2)}(e_kg^j)_{(2)}&=\sum_{rs=h}\sum_{cd=k}\langle e_rg^i,e_c g^j\rangle (e_sg^i)(e_d g^j)\\
                &=\sum_{rs=h}\sum_{cd=k}\delta_{r,b^j}\delta_{c,b^{-i}}\lambda^{ij}(e_sg^i)(e_d g^j)\\
                &=\lambda^{ij}(e_{b^{-j}h}g^i)(e_{b^ik} g^j)\\
                &=\lambda^{ij}\delta_{b^{-j}h,kb^i}e_{b^{-j}h}g^{i+j}
\end{align*}
and
\begin{align*}
(e_kg^j)_{(1)}(e_hg^i)_{(1)}\langle (e_hg^i)_{(2)},(e_kg^j)_{(2)}\rangle &=\sum_{rs=h}\sum_{cd=k} (e_cg^j)(e_r g^i)\langle e_s g^i,e_d g^j\rangle\\
                &=\sum_{rs=h}\sum_{cd=k} (e_cg^j)(e_r g^i)\delta_{s,b^j}\delta_{d,b^{-i}}\lambda^{ij}\\
                &=\lambda^{ij} (e_{kb^i} g^j)(e_{hb^{-j}}g^i)\\
                &=\lambda^{ij}\delta_{kb^i,b^{-j}h}e_{kb^i}g^{i+j}\\
                &=\lambda^{ij}\delta_{b^{-j}h,kb^i}e_{b^{-j}h}g^{i+j},
\end{align*}
we have the equation \eqref{eq2.1}.
\end{proof}
\subsection{Quasitriangular structures on $\mathscr{B}_{\lambda}$ and $\mathscr{B}_{\lambda}^\ast$.}

Let $0 \leq \lambda< p-1$. We will consider the quasitriangular structures on $\mathscr{B}_{\lambda}$ and $\mathscr{B}_{\lambda}^\ast$ in this subsection. The following result seems has its own interesting.

\begin{lemma}\label{lem3.9}
Let $(H,R)$ be a quasitriangular Hopf algebra over $\Bbbk$. Assume $G(H)$ is an abelian group and all proper Hopf subalgebras are trivial. Then $H_R=H$ or $H_R \subseteq \Bbbk G(H) $. Moreover, we have $H^{\ast cop}\cong H$ if $H_R=H$.
\end{lemma}

\begin{proof}
Assume $H_R\neq H$, then we will show $H_R \subseteq \Bbbk G(H) $. Since $H_l\subseteq H_R$, we have $H_l\neq H$. By assumption, we know $H_l$ is trivial, i.e it is group algebra or the dual of group algebra. Since $H_l^{\ast cop}\cong H_r$, we obtain $H_l$ or $H_r$ is group algebra. Assume $H_l$ is group algebra. By definition, we have $H_l\subseteq \Bbbk G(H)$. Because $G(H)$ is an abelian group, we know $H_l$ is commutative. This implies $H_r$ is also group algebra. Thus $H_r\subseteq \Bbbk G(H)$. Therefore, $H_R \subseteq \Bbbk G(H) $. Similarly, if $H_r$ is group algebra, then we also have $H_R \subseteq \Bbbk G(H) $. So we have shown $H_R \subseteq \Bbbk G(H)$ when $H_R\neq H$.

Assume $H_R=H$. We will prove $H^{\ast cop}\cong H$ in this case. If $H_l\neq H$, then we can repeat the previous process and get $H_R \subseteq \Bbbk G(H)$. This implies $\Bbbk G(H)=H$. Since $G(H)$ is an abelian group, we know $H^{\ast cop}\cong H$. If $H_l=H$, then we have $H_r=H$. Because $H_l^{\ast cop}\cong H_r$, we also obtain $H^{\ast cop}\cong H$.
\end{proof}
With the help of this lemma, we can discuss the quasitriangular structures on $\mathscr{B}_{\lambda}$ (resp. $\mathscr{B}_{\lambda}^\ast$), which are defined in Example \ref{ex2.1.6} (resp. Example \ref{ex2.1.7}). At first, we find that
\begin{proposition}\label{pro3.10}
The Hopf algebras $\mathscr{B}_\lambda^\ast$ admit no quasitriangular structure, where $0 \leq \lambda< p-1$.
\end{proposition}

\begin{proof}
Let $0 \leq \lambda< p-1$. Assume $R$ is a quasitriangular structure on $\mathscr{B}_\lambda^\ast$. We first show that $(\mathscr{B}_\lambda^\ast,R)$ satisfies the conditions of Lemma \ref{lem3.9}. By Remark \ref{rem1.1}, we get $G(\mathscr{B}_\lambda^\ast)$ is an abelian group. Due to the dimension reason and Lemma \ref{thm2.4}, we know all proper Hopf subalgebras of $\mathscr{B}_\lambda^\ast$ are trivial. Thus $\mathscr{B}_\lambda^\ast$ satisfies the conditions of Lemma \ref{lem3.9}. By Remark \ref{rem1.1} again, we have $G(\mathscr{B}_{\lambda}^{cop})\not \cong G(\mathscr{B}_{\lambda}^\ast)$ and hence $\mathscr{B}_{\lambda}^{cop}\not \cong \mathscr{B}_{\lambda}^\ast$. Thus $R\in \Bbbk G(\mathscr{B}_{\lambda}^\ast)\otimes \Bbbk G(\mathscr{B}_{\lambda}^\ast)$.

If $\lambda\neq 0$, then we have $\Bbbk G(\mathscr{B}_{\lambda}^\ast)=\{e_{g^i}|\;0\leq i \leq p-1\}$ by Remark \ref{rem1.1}. Since $R\in \Bbbk G(\mathscr{B}_{\lambda}^\ast)\otimes \Bbbk G(\mathscr{B}_{\lambda}^\ast)$, we can assume that $R=\sum_{r,s\in G}w(r,s)e_r\otimes e_s$, where $G=G(\mathscr{B}_{\lambda}^\ast)$ and $w$ is a bicharacter on $G$. Because $R$ is a quasitriangular structure, we have $\Delta^{cop}(a)R=R\Delta(a)$. Recall the Example \ref{ex2.1.7}, we get $\Delta(a)=\sum_{r,s\in G}e_r (s\triangleright a)\otimes e_s a$. Thus we have
\begin{align}
\Delta^{cop}(a)R&=[\sum_{r,s\in G}e_s a \otimes e_r (s\triangleright a)]R\\
&=[\sum_{r,s\in G}e_s a \otimes e_r (s\triangleright a)][\sum_{r,s\in G}w(r,s)e_r\otimes e_s]\\
&=\sum_{r,s\in G}w(r,s)e_r a\otimes e_s (r\triangleright a)
\end{align}
and $R\Delta(a)=\sum_{r,s\in G}w(r,s)e_r (s\triangleright a)\otimes e_s a $. Since $e_g a\otimes e_g (g\triangleright a)$ appear in $\Delta^{cop}(a)R$ while not in $R\Delta(a)$, we know $\Delta^{cop}(a)R\neq R\Delta(a)$. But this contradicts with the previous conclusion, so $\mathscr{B}_\lambda^\ast$ admits no quasitriangular structure.

If $\lambda=0$. By Remark \ref{rem1.1}, we get $\{ e_{g^i} b^j|\;0\leq i \leq p-1, 0\leq j \leq q-1\}$ is a linear basis of $\Bbbk G(\mathscr{B}_{0}^\ast)$. Because $R\in \Bbbk G(\mathscr{B}_{0}^\ast)\otimes \Bbbk G(\mathscr{B}_{0}^\ast)$, we can assume that $R=\sum_{0\leq i,j \leq p-1, 0\leq k,l \leq q-1} \lambda_{i,j}^{k,l}e_{g^i}b^k\otimes e_{g^j}b^l$, where $\lambda_{i,j}^{k,l}\in \Bbbk$. Because $R$ is a quasitriangular structure, we have $\Delta^{cop}(a)R=R\Delta(a)$. By Example \ref{ex2.1.7}, we get $\Delta(a)=\sum_{0\leq i,j \leq p-1}e_{g^i}a^{m^j}\otimes e_{g^j}a$. If we mimic the previous process for $\mathscr{B}_{\lambda}^\ast$, $\lambda\neq 0$, then we get
\begin{align}
\label{eq2.4} \Delta^{cop}(a)R=\sum\limits_{\begin{subarray}{l}  0\leq i,j \leq p-1  \\
                             0\leq k,l \leq q-1  \\
        \end{subarray}}\lambda_{i,j}^{k,l}e_{g^i}ab^k\otimes e_{g^j}a^{m^i}b^l.
\end{align}
Similarly, we have
\begin{align}
\label{eq2.5} R\Delta(a)=\sum\limits_{\begin{subarray}{l}  0\leq i,j \leq p-1  \\
                             0\leq k,l \leq q-1  \\
        \end{subarray}}\lambda_{i,j}^{k,l}e_{g^i}b^ka^{m^j}\otimes e_{g^j}b^la.
\end{align}
By definition, we get $ba=\sum_{0\leq i \leq p-1}\zeta_i^{m^{(\lambda+1)i}}e_{g^i}(ab)$. By induction way, we have $e_{g^i}b^ka^{m^j}=\zeta_i^{km^{(\lambda+1)i}m^j}e_{g^i}a^{m^j}b^k$ and $ e_{g^j}b^la=\zeta_j^{m^{(\lambda+1)j}l}e_{g^j}ab^l$. By combining these equalities with the equation \eqref{eq2.5}, we get
\begin{align}
\label{eq2.6} R\Delta(a)=\sum\limits_{\begin{subarray}{l}  0\leq i,j \leq p-1  \\
                             0\leq k,l \leq q-1  \\
        \end{subarray}}\lambda_{i,j}^{k,l}\zeta_i^{km^{(\lambda+1)i}m^j}\zeta_j^{m^{(\lambda+1)j}l} e_{g^i}a^{m^j}b^k\otimes e_{g^j}ab^l.
\end{align}
If $1\leq i\leq p-1$ and $0\leq j\leq p-1$, $0\leq k,l \leq q-1$, then $e_{g^i}ab^k\otimes e_{g^j}a^{m^i}b^l$ will appear in \eqref{eq2.4} while not in \eqref{eq2.6}. Since $\{e_{g^i}a^kb^l|\;0\leq i \leq p-1, 0\leq k,l \leq q-1\}$ is a linear basis of $\mathscr{B}_{0}$ and $\Delta^{cop}(a)R=R\Delta(a)$, we get the coefficient of $e_{g^i}ab^k\otimes e_{g^j}a^{m^i}b^l$ in \eqref{eq2.5} is $0$, i.e we have $\lambda_{i,j}^{k,l}=0$ for $i\neq 0$. Thus $R=\sum_{0\leq j \leq p-1, 0\leq k,l \leq q-1} \lambda_{0,j}^{k,l}e_{1}b^k\otimes e_{g^j}b^l$. Directly we have $(e_a\otimes e_1)R=0$, and hence $R$ is not invertible. This implies $R$ is not quasitriangular structure, so $\mathscr{B}_0^\ast$ admits no quasitriangular structure.
\end{proof}

Secondly, we can describe all possible quasitriangular structures on $\mathscr{B}_{\lambda}$. Before this, we define $\eta(h,k,g^i):=\tau(h,k,g^i)\tau(k,h,g^i)^{-1}$ for $h,k
\in \mathbb{Z}_q\times \mathbb{Z}_q, i\geq 0$. Then we have
\begin{proposition}\label{pro3.11}
All the quasitriangular structures on $\mathscr{B}_{\lambda}$ are given by $$R=\sum_{r,s\in G}w(r,s)e_r \otimes e_s,$$ where $w$ is a bicharacter on $\mathbb{Z}_q \times \mathbb{Z}_q$ such that the following conditions
         \begin{align*}
w(a,a)=w(a^m,a^m), \;w(a,b)=w(a^m,b^{m^\lambda})\eta(a^m,b^{m^\lambda},g),\\
w(b,a)=w(b^{m^\lambda},a^m)\eta(b^{m^\lambda},a^m,g),\;w(b,b)=w(b^{m^\lambda},b^{m^\lambda}).
\end{align*}
\end{proposition}
\begin{proof}
Assume $R$ is a quasitriangular structure on $\mathscr{B}_{\lambda}$. By Remark \ref{rem1.1}, we get $G(\mathscr{B}_{\lambda}^{\ast cop})\not \cong G(\mathscr{B}_{\lambda})$ and hence $\mathscr{B}_{\lambda}^{\ast cop}\not \cong \mathscr{B}_{\lambda}$. Due to the dimension reason and Theorem \ref{thm2.4}, we know all proper Hopf subalgebras of $\mathscr{B}_{\lambda}$ are trivial. By Lemma \ref{lem3.9}, we can suppose that $R=\sum_{r,s\in G}w(r,s)e_r \otimes e_s$, where $w$ is a bicharacter on $\mathbb{Z}_q \times \mathbb{Z}_q$. Since $R$ is a quasitriangular structure, we have $\Delta^{cop}(g)R=R\Delta(g)$. By definition, we get $\Delta(g)=\sum_{r,s\in G}\tau(r,s,g)e_r g\otimes e_s g$. Thus we have
\begin{align*}
\Delta^{cop}(g)R&=(\sum_{r,s\in G}\tau(s,r,g)e_r g\otimes e_s g)R\\
&=(\sum_{r,s\in G}\tau(s,r,g)e_r g\otimes e_s g)(\sum_{c,d\in G}w(c,d)e_c \otimes e_d)\\
&=\sum_{r,s,c,d\in G}\tau(s,r,g)w(c,d)e_r e_{c\triangleleft g^{-1}}g\otimes e_s e_{d\triangleleft g^{-1}}g\\
&=\sum_{r,s\in G}\tau(s,r,g)w(r\triangleleft g,s\triangleleft g)e_r g\otimes e_s g
\end{align*}
and $R\Delta(g)=\sum_{r,s\in G}\tau(r,s,g)w(r,s)e_r g\otimes e_s g$. Therefore $w(r\triangleleft g,s\triangleleft g)=w(r,s)\eta(r,s,g)$ for $r,s\in G$, where $\eta(r,s,g)=\tau(r,s,g)\tau(s,r,g)^{-1}$. By definition, we know $\tau(-,-,g)$ is 2-cocycle. Moreover, because $G$ is an abelian group, we get $\eta(-,-,g)$ is bicharacter on $G$. Thus $w(r\triangleleft g,s\triangleleft g)=w(r,s)\eta(r,s,g)$ for $r,s\in G$ if and only if they hold for generators of $G$, i.e they hold for $r,s\in \{a,b\}$. For convenience, we use the equivalent form of the equations $w(r\triangleleft g,s\triangleleft g)=w(r,s)\eta(r,s,g)$ for $r,s\in \{a,b\}$, which is $w(r,s)=w(r\triangleleft g^{-1},s\triangleleft g^{-1})\eta(r\triangleleft g^{-1},s\triangleleft g^{-1},g)$ for $r,s\in \{a,b\}$. And hence we have
\begin{align}
\label{eq2.2}w(a,a)=w(a^m,a^m), \;w(a,b)=w(a^m,b^{m^\lambda})\eta(a^m,b^{m^\lambda},g),\\
\label{eq2.3} w(b,a)=w(b^{m^\lambda},a^m)\eta(b^{m^\lambda},a^m,g),\;w(b,b)=w(b^{m^\lambda},b^{m^\lambda}).
\end{align}
Conversely, assume $w$ be a bicharacter on $G$ such that the equations \eqref{eq2.2}-\eqref{eq2.3}, then we can define $R=\sum_{r,s\in G}w(r,s)e_r \otimes e_s$. Due to the equations \eqref{eq2.2}-\eqref{eq2.3}, we get $\Delta^{cop}(g)R=R\Delta(g)$. Because $R$ is invertible and $\{e_h,g|\;h\in G\}$ generates $\mathscr{B}_{\lambda}$ as algebra, we know $\Delta^{cop}(x)R=R\Delta(x)$ for $x\in \mathscr{B}_{\lambda}$. Therefore $R$ is a quasitriangular structure on $\mathscr{B}_{\lambda}$.
\end{proof}

\textbf{Proofs of Theorems \ref{thm1.1} and \ref{thm1.2}.} Due to Corollary \ref{cor4.2} and Propositions \ref{pro3.7},\ref{pro3.8} and \ref{pro3.10},\ref{pro3.11}, we know Theorems \ref{thm1.1} and \ref{thm1.2} hold.\qed

\end{document}